\documentclass{amsart}\usepackage{ifthen}\newboolean{amsart}\setboolean{amsart}{true}

\DeclareMathOperator{\gap}{gap}

\newtheorem{theorem}{Theorem}
\newtheorem{lemma}[theorem]{Lemma}
\newtheorem{corollary}[theorem]{Corollary}

\begin{document}

\ifthenelse{\boolean{amsart}}{
\title[The rate of growth of $Q(n,\lceil
rn\rceil)$]{The rate of growth of the minimum clique size of graphs of given
order and chromatic number}
}{
\title{The rate of growth of the minimum clique size of graphs of given
order and chromatic number}
}

\ifthenelse{\boolean{amsart}}{
\author{Csaba Bir\'o}
\address[Csaba Bir\'o]{Department of Mathematics, University of Louisville, Louisville, KY 40292}
\email{csaba.biro@louisville.edu}
\author{Kris Wease}
\address[Kris Wease]{Department of Mathematics, University of Louisville, Louisville, KY 40292}
}{
\author{Csaba Bir\'o\footnote{Department of Mathematics, University of
Louisville, Louisville, KY 40292, \texttt{csaba.biro@louisville.edu}} \and
Kris Wease\footnote{Department of Mathematics, University of
Louisville, Louisville, KY 40292}}
\date{}
}

\ifthenelse{\boolean{amsart}}{}{\maketitle}

\begin{abstract}
Let $Q(n,c)$ denote the minimum clique number over graphs with $n$ vertices and
chromatic number $c$. We determine the rate of growth of the sequence
$\{Q(n,\lceil rn \rceil)\}_{n=1}^\infty$ for any fixed $0<r\leq 1$. We also give
a better upper bound for $Q(n,\lceil rn \rceil)$.
\end{abstract}

\ifthenelse{\boolean{amsart}}{\maketitle}{}

\section{Introduction}\label{S:intro}

Let $\omega(G),\alpha(G),$ and $\chi(G)$ denote the clique number, independence
number, and chromatic number, respectively, of a graph $G$. We will also use
$|G|$ to denote the number of vertices and $\|G\|$ to denote the number of
edges of $G$. Furthermore, let
\[
\omega(n,k)=\min\{\omega(G):|G|=n\text{ and }\alpha(G)\leq k\}
\]
the inverse Ramsey number.
Define
\[
Q(n,c)=\min\{\omega(G) : |G|=n \text{ and } \chi(G)=c\}.
\]
The goal of this research is to determine $Q(n,c)$ as exactly as possible.

Bir\'o, F\"uredi, and Jahanbekam \cite{Bir-Fur-Jah-13} gave an exact formula
for $Q(n,c)$ for the case when $c\geq (n+3)/2$ in terms of inverse Ramsey
numbers. They proved the following.
\begin{theorem}\label{thm:bfj}
For $n\geq 2k+3$
\[
Q(n,n-k)=n-2k+q(k)
\]
where
\[
q(k)=\min\sum_{i=1}^s(\omega(2k_i+1,2)-1)
\]
where the minimum is taken over positive integers $k_1,\ldots,k_s$ with
$k_1+\cdots+k_s=k$, and $s\leq 3$.
\end{theorem}

Liu \cite{Liu-12} determined the rate of growth of $Q(n,\lceil n/k\rceil)$ for
$k$ fixed positive integer, still, in terms of inverse Ramsey numbers.
He proved that $Q(n,\lceil n/k\rceil)=\Theta(\omega(n,k))$ for $k$ positive
integer. The natural question (also specifically posed by Liu) remained to
determine the rate of growth of the sequence in cases when $k$ is not an
integer. In this paper we provide the answer to this question proving the
following theorem.
\begin{theorem}\label{thm:main}
Fix $0<r\le1$ and let $k=\lfloor1/r\rfloor$. Then there exists $0<d_r\le
1$ such that for $n$ large enough
\[
d_r\omega(n,k)\le Q(n,\lceil rn\rceil)\le \omega(n,k).
\]
\end{theorem}

We go beyond these bounds in Section~\ref{S:upperbound}: we provide a stronger upper bound for $Q(n,\lceil
rn\rceil)$. We hope that the improved bound is close to the actual value, in fact it is plausible to believe that it is asymptotically correct.

A related line of research was done in \cite{Gya-Seb-Tro-12}, in which the
authors study the chromatic gap: $\gap(G)=\max\{ \chi(G)-\omega(G): |V(G)|=n \}$. The obvious relationship $\gap(n)= \max \{ c - Q(n,c)\}$ makes our questions slightly more general.

\section{Proof of the main theorem}
In the following proof, we generalize some of Liu's ideas to make it work
for arbitrary (non-integer) positive real numbers, though at the end the proof
is substantially different. Still, it is very interesting to note that the
jumps in the rate of growth happens when $r$ is a reciprocal of a positive
integer.

We will need the following simple lemma.
\begin{lemma}\label{l:subadditive}
For all $0<r\leq 1$, and $n,k$ positive integers, if $rn\geq k$, then
\[
\omega(\lceil
rn\rceil,k)\geq\frac{1}{\left\lceil\frac{1}{r}\right\rceil}\omega(n,k).
\]
\begin{proof}
Observe that $\omega$ is monotone and sub-additive in its first variable. Therefore
\[
\left\lceil\frac{1}{r}\right\rceil\omega(\lceil rn\rceil,k)
\geq \omega\left(\left\lceil\frac{1}{r}\right\rceil\lceil rn\rceil,k\right)
\geq \omega(n,k).
\]
\end{proof}
\end{lemma}

Now we will prove that $Q(n,\lceil rn\rceil)\leq \omega(n,k)$; we do this by
exhibiting a graph with $n$ vertices, chromatic number $\lceil rn\rceil$, and
clique number at most $\omega(n,k)$. Let $G$ be a Ramsey graph with $|G|=n$,
$\alpha(G)=k$, and $\omega(G)=\omega(n,k)$. Then
\[
\chi(G)\geq\left\lceil\frac{n}{k}\right\rceil=\left\lceil\frac{n}{\lfloor
1/r\rfloor}\right\rceil\geq\left\lceil\frac{n}{1/r}\right\rceil=\lceil r n
\rceil.
\]
Drop edges from $G$ until we get a subgraph $G'$ with $\chi(G')=\lceil
rn\rceil$. Then $|G'|=n$, and $\omega(G')\leq\omega(n,k)$.

Now we will prove the existence of the constant $d_r$.

Let $G$ be a graph with $|G|=n$ and $\chi(G)=\lceil rn\rceil$. In the first
step, we will show that there exists a constant $c_r$ (that only depends on
$r$), and an $H$ subgraph of $G$, such that $|H|\geq c_r n$, and $\alpha(H)\leq
k$. We will construct $H$ from $G$ by removing independent sets of size $k+1$,
as many as possible. In other words, Let $S$ be a largest collection of
disjoint independent sets of size $k+1$ in $G$, and let $H=G-S$.

From the maximality of $S$, it is clear that $\alpha(H)\leq k$. Let $t=|S|$.
Since $\chi(G)\leq t+|H|$, and $|H|=n-t(k+1)$, we have
\[
\lceil rn\rceil=\chi(G)\leq t+n-t(k+1)=n-tk.
\]
This implies $tk\leq n-\lceil rn\rceil\leq n-rn=(1-r)n$, so $t\leq (1-r)n/k$.
It follows that
\[
|H|\geq n-\frac{(1-r)n}{k}(k+1)=\left(\frac{k+1}{k}r-\frac{1}{k}\right)n
\]
Let $c_r=\frac{k+1}{k}r-\frac{1}{k}$. Recall that $k=\lfloor 1/r\rfloor$, so
$c_r$ is determined by $r$, and $r>1/(k+1)$; therefore also
$c_r>0$.

We established the existence of a subgraph $H$ with $|H|\geq c_r$ and
$\alpha(H)\leq k$. Then for large $n$,
\[
\omega(G)\geq \omega(H)\geq\omega(\lceil c_r
n\rceil,k)\geq\frac{1}{\left\lceil\frac{1}{c_r}\right\rceil}\omega(n,k),
\]
where the last inequality follows from Lemma~\ref{l:subadditive}. Hence we may
choose $d_r=1/\lceil 1/c_r\rceil$.\qed

Our constants $d_r$ provide improvements on Liu's constants in case $r$ is the
reciprocal of an integer. Indeed if $r=1/k$ for a $k$ integer and $k\to\infty$,
Liu's constants will exponentially converge to zero, while $c_r=1/k^2=r^2$.

It is also very interesting to note that as $r$ approaches the reciprocal of
integer from above, $c_r\to 0$. We tend to believe that this is just an
artifact of the proof, but it would be very interesting to see this question
settled one way or the other.

\section{Better upper bound}\label{S:upperbound}

In the previous section we only proved a weak bound for $Q(n,\lceil
rn\rceil)$, because that was all we needed to establish the rate of growth. But
that bound is certainly not optimal. In the following, we show how to get better
bounds in case $r$ is \emph{not} a reciprocal of an integer.

\begin{theorem}\label{thm:upperbound}
Let $0<r\leq 1$ such that $1/r$ is not an integer. Let $k=\lfloor 1/r\rfloor$, and let $m=n-k\lceil rn\rceil$,
$l=(k+1)\lceil rn\rceil-n$. Let $q(\beta,\alpha)=\min\sum\omega(\alpha
\beta_i,\alpha)$
where the minimum is taken over sums $\sum
\beta_i=\beta$ with $\beta_i>0$ integers. Then for large enough $n$,
\[
Q(n,\lceil rn\rceil)\leq q(l,k)+q(m,k+1).
\]
\end{theorem}

Before the proof, let us comment on the requirement on $1/r$. Notice that for
large $n$, we have $m>0$; in other words, for all $r$ that is not a reciprocal
of an integer there exists $N$ such that $n>N$ implies $n-k\lceil rn\rceil>0$.
Therefore, the quantity $q(m,k+1)$ in the statement is well-defined. If we do
not set the requirement on $r$, the statement breaks down at reciprocals of
integers due to some rounding problems. Note that, on the other hand, $l>0$ is
always true, because the rounding in that case works in our favor.

It may seem that the requirement on $r$ takes away from the power of the
theorem, but in fact if $r$ is close to the reciprocal of an integer, $m$ will
get close to zero, and then the statement of the theorem is hardly an
improvement on Theorem~\ref{thm:main}. In fact, it is expected that this
approach would not prove any better bounds for exact reciprocals of integers.

The motivation of the theorem is that we do not believe that the jumps in the
rate of growth proven in Theorem~\ref{thm:main} show the whole picture.
Between these jumps, far from reciprocals of integers, the upper bound can be
improved, as it is demonstrated by the theorem.

\begin{proof}[Proof of Theorem~\ref{thm:upperbound}]
We will exhibit a graph on $n$ vertices (for large $n$) with chromatic number $\lceil rn\rceil$
and clique number at most $q(l,k)+q(m,k+1)$. To do this, let $l_1,\ldots,l_a$
be the numbers that minimize $q(l,k)$, and let $m_1,\ldots,m_b$ be the
numbers that minimize $q(m,k+1)$. Let $L_1,\ldots,L_a$ be Ramsey graphs with
$|L_i|=kl_i$, $\alpha(L_i)\leq k$, and $\omega(L_i)=\omega(kl_i,k)$. Similarly,
let $M_1,\ldots,M_b$ Ramsey graphs with $|M_i|=(k+1)m_i$, $\alpha(M_i)\leq
k+1$, and $\omega(M_i)=\omega((k+1)m_i,k+1)$. Now construct $G$ by taking the
disjoint union of $L_1,\ldots,L_a,M_1,\ldots,M_b$, and add every edge between
any two of these components. Then clearly, $|G|=kl+(k+1)m=n$, and
\[
\chi(G)
=\sum_{i=1}^a\chi(L_i)+\sum_{j=1}^b\chi(M_j)
\geq \sum_{i=1}^a\frac{|L_i|}{k}+\sum_{j=1}^b\frac{|M_j|}{k+1}=l+m=\lceil
rn\rceil,
\]
furthermore
\[
\omega(G)=\sum_{i=1}^a\omega(L_i)+\sum_{j=1}^b\omega(M_j)=q(l,k)+q(m,k+1).
\]
Now apply the usual trick of dropping edges until the chromatic number is down
to $\lceil rn\rceil$ to get the example graphs.
\end{proof}

\begin{corollary} Let $0<r\leq 1$ such that $1/r$ is not an integer, and
$k,l,m$ defined as in Theorem~\ref{thm:upperbound}. Then
\[
Q(n,\lceil rn\rceil)\leq\omega(kl,k)+\omega((k+1)m,k+1).
\]
\end{corollary}
\begin{proof}
By the definition of the function $q(\beta,\alpha)$ from
Theorem~\ref{thm:upperbound}, we have
$q(\beta,\alpha)\leq \omega(\alpha\beta,\alpha)$, and the statement follows.
\end{proof}
Note that $Q(n,\lceil rn\rceil)\leq \omega(kl,k)+\omega((k+1)m,k)$ is a direct
consequence of Theorem~\ref{thm:main} and sub-additivity of $\omega$ in the
first variable. But the corollary does provide actual improvements over
Theorem~\ref{thm:main}, because the function $\omega(\cdot,\cdot)$ is monotone
decreasing in the second variable.

The corollary may be weaker than Theorem~\ref{thm:upperbound}, but it has the
advantage that it expresses the upper bound as the sum of only two inverse
Ramsey numbers, as opposed to a minimum over sums of inverse Ramsey numbers,
like the thorem does.

\section{Final notes}

The bound provided by Theorem~\ref{thm:upperbound} is almost certainly not
exact, because one can probably improve on it by just choosing sizes more
carefully for the Ramsey graphs $L_i$ and $M_j$. But perhaps the more
interesting problem that is left open is to establish an asymptotically correct
formula for the sequence $Q(n,\lceil rn\rceil)$. As we mentioned above, we
believe that the bounds in Section~\ref{S:upperbound} have a good chance to be
asymptotically correct, but proving it would probably require a good
understanding of certain restricted clique packings of graphs.

\section{Acknowledgment}

The authors would like to thank the unknown referees for valuable suggestions;
in particular on the simplification of the proof of the main theorem.

\bibliography{bib}
\bibliographystyle{amsplain}

\end{document}